\documentclass[a4paper,english,11pt]{article}

\usepackage[utf8]{inputenc}
\usepackage{amsmath,amsthm,amssymb,amsxtra}
\usepackage{mathrsfs}
\usepackage{pdfsync}
\usepackage{babel} 
\usepackage{latexsym} 
\usepackage{textcomp} 
\usepackage{bm}  
\usepackage{colortbl}
\usepackage{enumitem}

\newcommand*{\scrF}{\ensuremath{\mathscr{F}}} 
\newcommand*{\caB}{\ensuremath{\mathcal{B}}}	
\newcommand*{\caC}{\ensuremath{\mathcal{C}}}	
\newcommand*{\caF}{\ensuremath{\mathcal{F}}}	
\newcommand*{\caH}{\ensuremath{\mathcal{H}}}	
\newcommand*{\caO}{\ensuremath{\mathcal{O}}}	
\newcommand*{\caS}{\ensuremath{\mathcal{S}}}	

\newcommand*{\N}{\mathbb{N}}									
\newcommand*{\R}{\mathbb{R}}									
\newcommand*{\Rd}{{\mathbb{R}^d}}							

\newcommand*{\eps}{\varepsilon}								

\newcommand*{\E}{\mathbb{E}}									
\renewcommand*{\P}{\mathbb{P}}								

\newcommand*{\ii}{\mathrm{i}}
\newcommand{\Infkt}[1]{1_{ #1 }}							

\numberwithin{equation}{section}
\newtheorem{lemma}{Lemma}[section]
\newtheorem{theorem}[lemma]{Theorem}

\newtheorem{remark}[lemma]{Remark}
\newtheorem{example}[lemma]{Example}
\newtheorem{assumption}[lemma]{Assumptions}

\allowdisplaybreaks
\begin{document}

\begin{titlepage}
\null \vspace{0.5cm}
\begin{center}
{\Large\bf Non elliptic SPDEs and ambit fields:}\\
{\Large\bf existence of densities}
\medskip

by\\
\vspace{7mm}

\begin{tabular}{l@{\hspace{10mm}}l@{\hspace{10mm}}l}
{\sc Marta Sanz-Sol\'e}$\,^{(\ast)}$  &&{\sc Andr\'e S{\"u}\ss}$\,^{(\ast)}$\\
{\small Facultat de Matem\`atiques}          &&{\small Center for Advanced Study}\\
{\small University of Barcelona}  &&{\small Norwegian Academy}\\
{\small }  &&{\small of Sciences and Letters}\\ 
{\small Gran Via de les Corts Catalanes 585 }                                      &&{\small Drammensveien 78}\\
{\small 08007 Barcelona, Spain}        &&{\small 0271 Oslo, Norway}\\              
{\small e-mail: marta.sanz@ub.edu}       &&{\small e-mail: suess.andre@web.de}\\
{\small www.ub.edu/plie/Sanz-Sole}       &&{}
\null         
\end{tabular}
\end{center}

\vspace{1.5cm}

\noindent{\bf Abstract.} Relying on the method developed in \cite{debusscheromito}, we prove the existence of a density for two different examples of random fields indexed by $(t,x)\in(0,T]\times \Rd$. The first example consists of SPDEs with Lipschitz continuous coefficients driven by a Gaussian noise white in time and with a stationary spatial covariance, in the setting of \cite{dalang}. The density exists on the set where the nonlinearity $\sigma$ of the noise does not vanish. This complements the results in \cite{sanzsuess3} where $\sigma$ is assumed to be bounded away from zero. The second example is an ambit field
with a stochastic integral term having as integrator a L\'evy basis of pure-jump, stable-like type. 

\medskip

\medskip

\medskip


\vspace{1 cm}

\noindent
\footnotesize
{\begin{itemize} \item[$^{(\ast)}$] Supported by the grant MTM 2012-31192 from the \textit{Direcci\'on General de
Investigaci\'on Cient\'{\i}fica y T\'ecnica, Ministerio de Econom\'{\i}a y Competitividad, Spain.}
\end{itemize}}


\end{titlepage}

\newpage








\section{Introduction}

Malliavin calculus has proved to be a powerful tool for the study of questions concerning the probability laws of random vectors, ranging from its very existence to the study of their properties and applications.
Malliavin's probabilistic proof of H\"ormander's hypoellipticity theorem for differential operators in quadratic form provided the existence of an infinitely differentiable  
density with respect to the Lebesgue measure on $\R^m$ for the law at a fixed time $t>0$ of the solution to a stochastic differential equation (SDE) on $\R^m$ driven by a multi-dimension Brownian motion. The classical Malliavin's criterion for existence and regularity of densities (see, e.g. \cite{malliavin1}) requires strong regularity of the random vector $X$ under consideration. In fact, $X$ should be in the space $\mathbb{D}^\infty$, meaning that it belongs to Sobolev type spaces of any degree. As a consequence, many interesting examples are out
of the scope of the theory, for example, SDE with H\"older continuous coefficients, and others that will be mentioned throughout this introduction.

Recently, there have been several attempts to develop techniques to prove existence of density, under weaker regularity conditions than in the Malliavin's theory, but providing much less information on its properties. The idea is to avoid applying integration by parts, and use instead some approximation procedures.
A pioneer work in this direction is \cite{fournierprintems}, where the random vector $X$ is compared with a {\it good} approximation $X^\eps$ whose law is known. The proposal of the random vector $X^\eps$ is inspired in Euler numerical approximations and
 the comparison is done through their respective Fourier transforms. The method is illustrated with several one-dimensional examples os stochastic equations, all of them having in common that the {\it diffusion coefficient} is H\"older continuous and the drift term, a measurable function: SDEs, including cases of random coefficients, a stochastic heat equation with Neumann boundary conditions, and a SDE driven by a L\'evy process.

With a similar motivation, and relying also on the idea of {\it approximation}, A. Debussche and M. Romito prove in  \cite{debusscheromito} a useful criterion for the existence of density of
random vectors. In comparison with \cite{fournierprintems}, the result is formulated in an abstract form, it applies to multidimensional random vectors and provides additionally information on the
space where the density lives. The precise statement is given in Lemma \ref{lem:existencedensity} below. As an illustration of the method, \cite{debusscheromito} considers finite dimensional functionals of the solutions of the stochastic Navier-Stokes equations in dimension 3, and in \cite{debusschefournier} SDEs driven by stable-like L\'evy processes with H\"older continuous coefficients.
A similar methodology has been applied in  \cite{ballyclement1}, \cite{ballyclement2}, \cite{ballyfournier}. The more recent work \cite{ballycaramelino} applies interpolation arguments on Orlicz spaces to 
obtain absolute continuity results of finite measures. Variants of the criteria provide different types of properties of the density. The results are illustrated by diffusion processes with $\log$-H\"older coefficients and piecewise deterministic Markov processes.

Some of the methods developed in the references mentioned so far are also well-suited to the analysis of stochastic partial differential equations (SPDEs) defined by non-smooth differential operators. Indeed,
consider a class of SPDEs defined by
\begin{equation}\label{eq:SWE3d}
  Lu(t,x) = b(u(t,x)) + \sigma(u(t,x))\dot F(t,x), \ (t,x)\in(0,T]\times \Rd,
\end{equation}
with constant initial conditions, where $L$ denotes a linear differential operator, $\sigma, b: \R \to \R$,  and $F$ is a Gaussian noise, white in time with some spatial correlation (see Section \ref{s2} for
the description of $F$). Under some set of assumptions, \cite[Theorem 2.1]{sanzsuess3} establishes the existence of density for the random field solution of \eqref{eq:SWE3d} at any point $(t,x)\in(0,T]\times \Rd$, and also that the density belongs to some Besov space. The theorem applies for example to  the stochastic wave equation in any spatial dimensions $d\ge 1$.

The purpose of this work is to further illustrate the range of applications of Lemma \ref{lem:existencedensity} with two more examples. The first one is presented in the next Section \ref{s2} and complements the results of \cite{sanzsuess3}. In comparison with this reference, here we are able to remove the {\it strong ellipticity} property on the function $\sigma$, which is crucial in most of the applications of
Malliavin calculus to SPDEs (see \cite{sanzbook}), but the class of operators $L$ is more restrictive Nevertheless, Theorem \ref{t3.1}, applies for example to the stochastic heat equation in any spatial dimension and to the stochastic wave
equation with $d\le 3$. For the latter example, if $\sigma$, $b$ are smooth functions and $\sigma$ is bounded away from zero, existence and regularity of the density of $u(t,x)$ has been established in
\cite{quersanz1} and \cite{quersanz2}.

The second example, developed in Section \ref{s3}, refers to ambit fields driven by a class of L\'evy bases (see \eqref{eq:ambitfield}). Originally introduced in \cite{bns} in the context of modeling turbulence, ambit fields are stochastic processes indexed by time and space that are becoming popular and useful for the applications in mathematical finance among others. The expression \eqref{eq:ambitfield} has some similitudes with the mild formulation of \eqref{eq:SWE3d} (see \eqref{2.1}) and can  be also seen as an infinite dimensional extension of SDEs driven by L\'evy processes. We are not aware of
previous results on densities of random fields.

We end this introduction by quoting the definition of the Besov spaces relevant for this article as well as the existence of density criterion by \cite{debusscheromito}.

The spaces $B_{1,\infty}^s$, $s>0$,  can be defined as follows.
Let $f:\Rd\to\R$. For $x,h\in\Rd$ set $(\Delta^1_hf)(x)=f(x+h)-f(x)$. Then, for any $n\in\N$, $n\geq 2$, let
\[ (\Delta_h^nf)(x) = \big(\Delta^1_h(\Delta^{n-1}_hf)\big)(x) = \sum_{j=0}^n (-1)^{n-j}\binom{n}{j}f(x+jh). \]
For any $0<s<n$, we define the norm
\[ \|f\|_{B^s_{1,\infty}} = \|f\|_{L^1} + \sup_{|h|\leq 1} |h|^{-s}\|\Delta_h^nf\|_{L^1}. \]
It can be proved that for two distinct $n,n'>s$ the norms obtained using $n$ or $n'$ are equivalent. 
Then we define $B^s_{1,\infty}$ to be the set of $L^1$-functions with $\|f\|_{B^s_{1,\infty}}<\infty$.
We refer the reader to \cite{triebel} for more details.

In the following, we denote by $\caC^\alpha_b$ the set of bounded H\"older continuous functions of degree $\alpha$. The next Lemma establishes the criterion on existence of densities that we will apply in our examples.

\begin{lemma}\label{lem:existencedensity}
Let $\kappa$ be a finite nonnegative measure. Assume that there exist $0<\alpha\leq a<1$, $n\in\N$ and a constant $C_n$ such that for all $\phi\in\caC^\alpha_b$, and all $h\in\R$ with $|h|\leq1$,
\begin{equation}
	\bigg|\int_\R \Delta_h^n\phi(y)\kappa(d y)\bigg|\leq C_n\|\phi\|_{\caC^\alpha_b}|h|^a.
	\label{eq:existencedensity}
\end{equation}
Then $\kappa$ has a density with respect to the Lebesgue measure, and this density belongs to the Besov space $B^{a-\alpha}_{1,\infty}(\R)$.
\end{lemma}



\section{Nonelliptic diffusion coefficients}
\label{s2}
In this section we deal with SPDEs without the classical ellipticity assumption on the coefficient $\sigma$, i.e.\ $\inf_{x\in\Rd} |\sigma(x)|\geq c>0$. In the different context of SDEs driven by a L\'evy process, this situation was considered in \cite[Theorem 1.1]{debusschefournier}, assuming in addition that $\sigma$ is bounded. Here, we will deal with SPDEs in the setting of \cite{dalang} with not necessarily bounded coefficients $\sigma$. Therefore, the results will apply in particular to Anderson's type SPDEs ($\sigma(x)= \lambda x$, $\lambda \ne 0$).

 We  consider the class of SPDEs defined by \eqref{eq:SWE3d},
with constant initial conditions, where $L$ denotes a linear differential operator, and $\sigma, b: \R \to \R$. In the definition above, $F$ is a Gaussian noise, white in time with some spatial correlation. 

Consider the space of Schwartz functions on $\Rd$, denoted by $\caS(\Rd)$, endowed with the following inner product
\[ \langle \phi,\psi\rangle_{\caH} := \int_\Rd dy \int_\Rd \Gamma(dx) \phi(y)\psi(y-x), \]
where $\Gamma$ is a nonnegative and nonnegative definite tempered measure. Using the Fourier transform we can rewrite this inner product as
\[ \langle \phi,\psi\rangle_{\caH} = \int_\Rd  \mu(d\xi) \caF\phi(\xi)\overline{\caF\psi(\xi)}, \]
where $\mu$ is a nonnegative definite tempered measure with $\caF\mu = \Gamma$. Let 
$\caH := \overline{(\caS,\langle\cdot,\cdot\rangle_\caH)}^{\langle\cdot,\cdot\rangle_\caH}$,
and $\caH_T := L^2([0,T];\caH)$. It can be proved that $F$ is an isonormal Wiener process on $\caH_T$.

Let $\Lambda$ denote the fundamental solution to $Lu=0$ and assume that  $L$ is either a function or a non-negative measure of the form $\Lambda(t,dy)dt$ such that $\sup_{t\in[0,T]}\Lambda(t,\Rd)\le C_T<\infty$. We consider
\begin{align}
\label{2.1}
  u(t,x) &= \int_0^t\int_\Rd \Lambda(t-s,x-y)u(s,y)M(ds,dy) \notag\\
  & + \int_0^t\int_\Rd \Lambda(t-s,x-y)b(u(s,y))dydy,
\end{align}
as the integral formulation of \eqref{eq:SWE3d}, where $M$ is the martingale measure generated by $F$. In order for the stochastic integral in the previous equation to be well-defined, we need to assume that
\begin{equation*}
\int_0^T ds \int_\Rd \mu(d\xi) |\caF\Lambda(s)(\xi)|^2 < +\infty.
\end{equation*}
 According to \cite[Theorem 13]{dalang} (see also \cite{walsh}), equation \eqref{2.1} has a unique random field solution $\{u(t,x);, (t,x)\in[0,T]\times \Rd\}$ which has a spatially stationary law (this is a consequence of the \emph{S}-property in \cite{dalang}), and for all $p\geq2$
\[ \sup_{(t,x)\in[0,T]\times\Rd} \E\big[|u(t,x)|^p\big] < \infty. \]

We will prove the following result on the existence of a density.

\begin{theorem}
\label{t3.1}
The hypotheses are as in \cite[Theorem 13]{dalang}. Moreover, we assume that
\begin{align*} 
ct^\gamma \leq&\int_0^t ds \int_\Rd\mu(d\xi) |\caF\Lambda(s)(\xi)|^2 \leq Ct^{\gamma_1},\\
&\int_0^t ds |\caF\Lambda(s)(0)|^2 \leq Ct^{\gamma_2}, 
\end{align*}
 for some $\gamma, \gamma_1, \gamma_2>0$ and positive constants $c$ and $C$. Suppose also that there exists $\delta>0$ such that
  \begin{equation}
  \label{eq:eps^delta}
		\E\big[|u(t,0)-u(s,0)|^2\big] \leq C|t-s|^\delta,
  \end{equation}
  for any $s,t\in[0,T]$ and some constant $C>0$, and that 
   \begin{equation*}
   \bar\gamma := \frac{\min\{\gamma_1,\gamma_2\} + \delta}{\gamma} > 1. 
   \end{equation*}
   Fix $(t,x)\in(0,T]\times \R^d$. Then, the probability law of $u(t,x)$ has a density $f$ on the set $\{y\in\R;\sigma(y)\ne 0\}$. In addition, there exists $n\ge 1$ such that the function 
  $y\mapsto \vert \sigma(y)\vert^n f(y)$ belongs to the Besov space $B_{1,\infty}^\beta$, with $\beta\in(0,\bar\gamma-1)$.
  
 \begin{proof}
 We will apply Lemma \ref{lem:existencedensity}. Since the law of $u(t,x)$ is stationary in space, we can take $x=0$. Consider the measure 
 \begin{equation*}
 \kappa(dy) = |\sigma(y)|^n\left(P\circ u(t,0)^{-1}\right)(dy). 
 \end{equation*}
 We define the following approximation of $u(t,0)$. Let for $0<\eps<t$
	\begin{equation}\label{eq:addueps} 
		u^\eps(t,0) = U^\eps(t,0) + \sigma(u(t-\eps,0))\int_{t-\eps}^t\int_\Rd \Lambda(t-s,-y)M(d s,d y),
	\end{equation}
where
\begin{align*}
	U^\eps(t,0) = & \int_0^{t-\eps}\int_\Rd \Lambda(t-s,-y)\sigma(u(s,y))M(ds,dy)\\
			&	+ \int_0^{t-\eps}\int_\Rd \Lambda(t-s,-y)b(u(s,y))dyds \\
			& + b(u(t-\eps,0))\int_{t-\eps}^t\int_\Rd \Lambda(t-s,-y)dyds. 
\end{align*}
Applying the triangular inequality, we have
\begin{align}\label{eq:proof31}
  \bigg|\int_\R\Delta_h^n\phi(y)\kappa(dy)\bigg| 
  & = \big|\E\big[|\sigma(u(t,0))|^n\Delta_h^n\phi(u(t,0))\big]\big| \notag\\
  & \leq \big|\E\big[(|\sigma(u(t,0))|^n-|\sigma(u(t-\eps,0))|^n)\Delta_h^n\phi(u(t,0))\big]\big| \notag\\
  & \quad + \big|\E\big[|\sigma(u(t-\eps,0))|^n(\Delta_h^n\phi(u(t,0)) - \Delta_h^n\phi(u^\eps(t,0)))\big]\big| \notag\\
  & \quad + \big|\E\big[|\sigma(u(t-\eps,0))|^n\Delta_h^n\phi(u^\eps(t,0))\big]\big|.
\end{align}

Remember that $\|\Delta_h^n\phi\|_{\caC^\alpha_b}\leq C_n\|\phi\|_{\caC^\alpha_b}$. Consequently, 
\[ |\Delta_h^n\phi(x)| = |\Delta_h^{n-1}\phi(x) - \Delta_h^{n-1}\phi(x+h)| \leq C_{n-1}\|\phi\|_{\caC^\alpha_b}|h|^\alpha, \]
Using this fact, the first term on the right-hand side of the inequality in \eqref{eq:proof31} can be bounded as follows:
\begin{align}\label{eq:bounddeltahn2}
  \big|\E\big[ & (|\sigma(u(t,0))|^n - |\sigma(u(t-\eps,0))|^n)\Delta_h^n\phi(u(t,0))\big]\big| \notag\\
  & \leq C_n\|\phi\|_{\caC^\alpha_b}|h|^\alpha\E\big[\big||\sigma(u(t,0))|^n-|\sigma(u(t-\eps,0))|^n\big|\big].
\end{align}
Apply the equality $x^n-y^n = (x-y)(x^{n-1}+x^{n-2}y + \ldots + xy^{n-2}+y^{n-1})$ along with the Lipschitz continuity of $\sigma$ and H\"older's inequality, to obtain 
\begin{align}\label{eq:boundpolynomial}
  & \E\big[\big||\sigma(u(t,0))|^n-|\sigma(u(t-\eps,0))|^n\big|\big] \notag\\
  & \leq \E\bigg[ \big|\sigma(u(t,0))-\sigma(u(t-\eps,0))\big|\sum_{j=0}^{n-1} |\sigma(u(t,0))|^j|\sigma(u(t,0))|^{n-1-j}\bigg] \notag\\
  & \leq C\Big(\E\big[\big|u(t,0) - u(t-\eps,0)\big|^2\big]\Big)^{1/2} \bigg(\E\bigg[\bigg(\sum_{j=0}^{n-1} |\sigma(u(t,0))|^j|\sigma(u(t,0))|^{n-1-j}\bigg)^2\bigg]\bigg)^{1/2} \notag\\
  & \leq C_n\Big(\E\big[\big|u(t,0) - u(t-\eps,0)\big|^2\big]\Big)^{1/2} \notag\\
  & \leq C_n\eps^{\delta/2},
\end{align}
where we have used that $\sigma$ has linear growth, also that $u(t,0)$ has finite moments of any order and \eqref{eq:eps^delta}.
Thus,
\begin{equation}
\label{3.2}
\big|\E\big[(|\sigma(u(t,0))|^n - |\sigma(u(t-\eps,0))|^n)\Delta_h^n\phi(u(t,0))\big]\big| \leq C_n\|\phi\|_{\caC^\alpha_b}|h|^\alpha\eps^{\delta/2}. 
\end{equation}

With similar arguments,
\begin{align}
\label{3.3}
  & \big|\E\big[|\sigma(u(t-\eps,0))|^n(\Delta_h^n\phi(u(t,0)) - \Delta_h^n\phi(u^\eps(t,0)))\big]\big|\notag \\
  & \leq C_n\|\phi\|_{\caC^\alpha_b}\E\big[|\sigma(u(t-\eps,0))|^n|u(t,0)-u^\eps(t,0)|^\alpha\big] \notag \\
  & \leq C_n\|\phi\|_{\caC^\alpha_b}\big(\E\big[|u(t,0)-u^\eps(t,0)|^2\big]\big)^{\alpha/2}\big(\E\big[|\sigma(u(t-\eps,0))|^{2n/(2-\alpha)}\big]\big)^{1-\alpha/2} \notag \\
  & \leq C_n\|\phi\|_{\caC^\alpha_b}\eps^{\delta\alpha/2}\big(g_1(\eps) + g_2(\eps)\big)^{\alpha/2},
\end{align}
where in the last inequality we have used the upper bound stated in \cite[Lemma 2.5]{sanzsuess3}. It is very easy to adapt the proof of this lemma to the context of this section.
Note that the constant $C_n$ in the previous equation does not depend on $\alpha$ because
\begin{align*}
  \big(\E\big[|\sigma(u(t-\eps,0))|^{2n/(2-\alpha)}\big]\big)^{1-\alpha/2} 
  & \leq \big(\E\big[(|\sigma(u(t-\eps,0))|\vee1)^{2n}\big]\big)^{1-\alpha/2} \\
  & \leq \E\big[|\sigma(u(t-\eps,0))|^{2n}\vee1\big].
\end{align*}

Now we focus on the third term on the right-hand side of the inequality in \eqref{eq:proof31}.  Let $p_\eps$ denote the density of the zero mean Gaussian random variable
$\int_{t-\eps}^t\int_\Rd\Lambda(t-s,-y)M(ds,dy)$, which is independent of the $\sigma$-field $\scrF_{t-\eps}$ and has variance
\begin{equation*}
g(\eps):=\int_0^\eps ds \int_{\Rd} \mu(d\xi) |\caF\Lambda(s)(\xi)|^2 \geq C\eps^{\gamma}.
\end{equation*} 
In the decomposition \eqref{eq:addueps}, the random variable $U^{\eps}(t,0)$ is $\scrF_{t-\eps}$-measurable. Then, by conditioning with respect to $\scrF_{t-\eps}$ and using a change of variables, we obtain
\begin{align*}
	& \left|\E\left[|\sigma(u(t-\eps,0))|^n\Delta_h^n\phi(u^\eps(t,0))\right]\right| \\
	&=\big|\E\big[\E\big[1_{\{\sigma(u(t-\eps,0))\ne 0\}}|\sigma(u(t-\eps,0))|^n\Delta_h^n\phi(u^\eps(t,0))\big|\scrF_{t-\eps}\big]\big]\big|\\
  & = \bigg|\E\bigg[1_{\{\sigma(u(t-\eps,0))\ne 0\}}\int_\R |\sigma(u(t-\eps,0))|^n\Delta_h^n\phi(U_t^\eps + \sigma(u(t-\eps,0))y)p_{\eps}(y)dy\bigg]\bigg| \\
  & = \bigg|\E\bigg[1_{\{\sigma(u(t-\eps,0))\ne 0\}}\int_\R |\sigma(u(t-\eps,0))|^n\\
  &\qquad \times\phi(U_t^\eps + \sigma(u(t-\eps,0))y)\Delta_{-\sigma(u(t-\eps,0))^{-1}h}^n p_{\eps}(y)dy\bigg]\bigg| \\
  & \leq \|\phi\|_\infty \E\bigg[1_{\{\sigma(u(t-\eps,0))\ne 0\}}|\sigma(u(t-\eps,0))|^n\int_\R\big|\Delta_{-\sigma(u(t-\eps,0))^{-1}h}^np_{\eps}(y)\big|dy\bigg].
\end{align*}  

On the set $\{\sigma(u(t-\eps,0))\ne 0\}$, the integral in the last term can be bounded as follows, 
\begin{align*}
	\int_\R\big|\Delta_{-\sigma(u(t-\eps,0))^{-1}h}^np_{\eps}(y)\big|dy
  & \leq C_n |\sigma(u(t-\eps,0))|^{-n}|h|^n \|p^{(n)}_{\eps}\|_{L^1(\R)} \\
  & \leq C_n |\sigma(u(t-\eps,0))|^{-n}|h|^n g(\eps)^{-n/2},
\end{align*}  
where we have used the property $\Vert \Delta_h^nf\Vert_{L^1(\R)}\le C_n |h|^n \Vert f^{(n)}\Vert_{L^1(\R)}$, and also that $\|p^{(n)}_{\eps}\|_{L_1} = (g(\eps))^{-n/2}\leq C_n\eps^{-n\gamma/2}$
(see e.g. \cite[Lemma 2.3]{sanzsuess3}).

Substituting this into the previous inequality yields
\begin{equation}
\label{3.4}
\big|\E\big[|\sigma(u(t-\eps,0))|^n\Delta_h^n\phi(u^\eps(t,0))\big]\big| \leq C_n\|\phi\|_{\caC^\alpha_b}|h|^n \eps^{-n\gamma/2}, 
\end{equation}
because $\|\phi\|_{\infty}\leq\|\phi\|_{\caC^\alpha_b}$.

With \eqref{eq:proof31}, \eqref{3.2}, \eqref{3.3}, \eqref{3.4}, we have
\begin{align}
\label{3.5}
  & \bigg|\int_\R\Delta_h^n\phi(y)\kappa(dy)\bigg|\notag \\
  & \leq C_n\|\phi\|_{\caC^\alpha_b}\Big(|h|^\alpha\eps^{\delta/2} + \eps^{\delta\alpha/2}\big(g_1(\eps) + g_2(\eps)\big)^{\alpha/2} + |h|^n \eps^{-n\gamma/2}\Big)\notag \\
  & \leq C_n\|\phi\|_{\caC^\alpha_b}\Big(|h|^\alpha\eps^{\delta/2} + \eps^{(\delta+\gamma_1)\alpha/2} + \eps^{(\delta+\gamma_2)\alpha/2} + |h|^n \eps^{-n\gamma/2}\Big)\notag \\
  & \leq C_n\|\phi\|_{\caC^\alpha_b}\Big(|h|^\alpha\eps^{\delta/2} + \eps^{\gamma\bar\gamma\alpha/2} + |h|^n \eps^{-n\gamma/2}\Big)
\end{align}

Let $\eps = \tfrac{1}{2}t|h|^\rho$, with $\rho = 2n/(\gamma n+\gamma\bar\gamma\alpha)$. With this choice, the last term in \eqref{3.5} is equal to
\begin{align*}
  C_n\|\phi\|_{\caC^\alpha_b}\Big(|h|^{\alpha+ \frac{n\delta}{\gamma(n+\bar\gamma\alpha)}} + |h|^{\frac{n\bar\gamma\alpha}{n+\bar\gamma\alpha}}\Big).
\end{align*}
Since $\gamma_1\le \gamma$, by the definition of $\bar\gamma$, we obtain 
\begin{equation*}
\bar\gamma-1 = \frac{\min\{\gamma_1,\gamma_2\}}{\gamma} + \frac{\delta}{\gamma} - 1 \leq \frac{\delta}{\gamma}. 
\end{equation*}
Fix $\zeta\in(0,\bar\gamma-1)$. We can choose $n\in\N$ sufficiently large and $\alpha$ sufficiently close to $1$, such that
\begin{equation*}
 \alpha+ \frac{n\delta}{\gamma(n+\bar\gamma\alpha)} > \zeta + \alpha\quad\text{and}\quad\frac{n\bar\gamma\alpha}{n+\bar\gamma\alpha}>\zeta+\alpha.  
 \end{equation*}
This finishes the proof of the theorem.
\end{proof}
\end{theorem}

\begin{remark}
\begin{enumerate}[label=(\roman{enumi}),ref=(\roman{enumi})]
\item Assume that $\sigma$ is bounded from above but not necessary bounded away from zero. 
Following the lines of the proof of Theorem \ref{t3.1} we can also show the existence of a density without assuming the existence of moments of $u(t,x)$ of order higher than $2$. This applies in particular to SPDEs whose fundamental solutions are general distributions as treated in \cite{conusdalang}, extending the result on absolute continuity given in \cite[Theorem 2.1]{sanzsuess3}.
\item Unlike \cite[Theorem 2.1]{sanzsuess3}, the conclusion on the space to which the density belongs is less precise. We do not know whether the order $\bar\gamma-1$ is optimal.
\end{enumerate}
\end{remark}


\section{Ambit random fields}
\label{s3}
In this section we prove the absolute continuity of the law of a random variable generated by an ambit field at a fixed point $(t,x)\in[0,T]\times\Rd$. The methodology we use is very much inspired by \cite{debusschefournier}. Ambit fields where introduced in \cite{bns} with the aim of studying turbulence flows, see also the survey papers \cite{benth,podolskij}. They are stochastic processes indexed by $(t,x)\in[0,T]\times\Rd$ of the form
\begin{equation}\label{eq:ambitfield}
	X(t,x) = x_0 + \iint_{A_t(x)} g(t,s;x,y)\sigma(s,y)L(ds,dy) + \iint_{B_t(x)} h(t,s;x,y)b(s,y)dyds,
\end{equation}
where $x_0\in\R$, $g,h$ are deterministic functions subject to some integrability and regularity conditions, $\sigma,b$ are stochastic processes, $A_t(x),B_t(x)\subseteq[0,t]\times\Rd$ are the so-called ambit sets. The stochastic process $L$ is a {\it L\'evy basis} on the Borel sets $\caB([0,T]\times\Rd)$. More precisely, for any $B\in\caB([0,T]\times\Rd)$ the random variable $L(B)$ has an infinitely divisible distribution; 
given $B_1,\ldots,B_k$ disjoint sets of $B\in\caB([0,T]\times\Rd)$, the random variables $L(B_1),\ldots,L(B_k)$ are independent; and
for any sequence of disjoint sets $(A_j)_{j\in\N}\subset \caB([0,T]\times\Rd)$,
 \[ L(\cup_{j=1}^\infty A_j) = \sum_{j=1}^\infty L(A_j), \quad \P\text{-almost surely}. \]
Throughout the section, we will consider the natural filtration generated by $L$, i.e.\ for all $t\in[0,T]$,
\[ \scrF_t := \sigma(L(A); A\in[0,t]\times\Rd, \lambda(A)<\infty). \]
 
For deterministic integrands, the stochastic integral in \eqref{eq:ambitfield} is defined as in \cite{rajputposinski}. In the more general setting 
of \eqref{eq:ambitfield}, one can use the theory developed in \cite{chongkluppelberg}. We refer the reader to these references for the specific required hypotheses on $g$ and $\sigma$.

The class of L\'evy bases considered in this section are described by infinite divisible distributions of pure-jump, stable-like type. More explicitly, as in 
\cite[Proposition 2.4]{rajputposinski}, we assume that for any $B\in\caB([0,T]\times\Rd)$,
\[ \log\E\big[\exp(\ii\xi L(B))\big] = \int_{[0,T]\times\Rd} \lambda(ds,dy)\int_\R \rho_{s,y}(dz) \big(\exp(\ii\xi z - 1 - \ii\xi z\Infkt{[-1,1]}(z))\big), \]
where $\lambda$ is termed the {\it control measure} on the state space and $(\rho_{s,y})_{(s,y)\in [0,T]\times\Rd}$ is a family of L\'evy measures satisfying
\[ \int_\R \min\{1,z^2\}\rho_{s,y}(dz) = 1,\ \lambda-{\text{a.s.}}.\]

Throughout this section, we will consider the following set of assumptions on $(\rho_{s,y})_{(s,y)\in [0,T]\times\Rd}$ and on $\lambda$.

\begin{assumption}\label{ass:stable}
  Fix $(t,x)\in[0,T]\times\Rd$ and $\alpha\in(0,2)$, and for any $a>0$ let $\caO_a:=(-a,a)$. Then,
  \begin{enumerate}[label=(\roman{enumi}),ref=(\roman{enumi})]
    \item\label{itm:stablei} for all $\beta\in[0,\alpha)$ there exists a nonnegative function $C_\beta\in L^1(\lambda)$ such that for all $a>0$, 
    $$\int_{(\caO_a)^c} |z|^\beta \rho_{s,y}(dz)\leq C_\beta(s,y)a^{\beta-\alpha},  \ \lambda-{\text{a.s.}};$$
    \item\label{itm:stableii} there exists a non-negative function $\bar C\in L^1(\lambda)$ such that for all $a>0$, 
    $$\int_{\caO_a}|z|^2\rho_{s,y}(dz)\leq \bar C(s,y)a^{2-\alpha},  \ \lambda-{\text{a.s.}};$$
    \item\label{itm:stableiii} there exists a nonnegative function $c\in L^1(\lambda)$ and $r>0$ such that for all $\xi\in\R$ with $|\xi|>r$, 
        \[ \int_\R \big(1-\cos(\xi z)\big)\rho_{s,y}(dz) \geq c(s,y)|\xi|^\alpha, \  \lambda-{\text{a.s.}}.\]
\end{enumerate}
\end{assumption}

\begin{example}
Let 
$\rho_{s,y}(dz) = c_1(s,y)\Infkt{\{z>0\}}z^{-\alpha-1}dz + c_{-1}(s,y)\Infkt{\{z<0\}}|z|^{-\alpha-1}dz$, 
$(s,y)\in [0,T]\times\Rd$, and assume that $c_1, c_{-1}\in L^1(\lambda)$. 
This corresponds to stable distributions (see \cite[Lemma 3.7]{rajputposinski}). One can check that Assumptions \ref{ass:stable} are satisfied with 
 $C=\bar C=c_1 \vee c_{-1}$, and  $c=c_1\wedge c_{-1}$.
\end{example}


\begin{assumption}\label{ass:stable2}
\noindent{\bf (H1)} 
We assume that the deterministic functions $g,h:\{0\leq s<t\leq T\}\times\Rd\times\Rd\to\R$ and the stochastic processes $(\sigma(s,y);(s,y)\in[0,T]\times\Rd)$, $(b(s,y);(s,y)\in[0,T]\times\Rd)$ are such that the integrals on the right-hand side of \eqref{eq:ambitfield} are well-defined (see the conditions in \cite[Theorem 2.7]{rajputposinski} and \cite[Theorem 4.1]{chongkluppelberg}). We also suppose that for any $y\in\Rd$, $p\in[2,\infty)$ we have $\sup_{s\in[0,T]}\E[|\sigma(s,y)|^p]<\infty$.

\noindent {\bf (H2)} 
Let $\alpha$ be as in Assumptions \ref{ass:stable}. There exist $\delta_1,\delta_2>0$ such that for some $\gamma\in(\alpha,2]$ and, if $\alpha\ge 1$, for all $\beta\in[1,\alpha)$, or for $\beta=1$, if $\alpha <1$, 
\begin{align}\label{eq:hoeldercty}
  \E\big[|\sigma(t,x)-\sigma(s,y)|^\gamma\big] & \leq C_\gamma(|t-s|^{\delta_1\gamma} + |x-y|^{\delta_2\gamma}), \\
  \E\big[|b(t,x)-b(s,y)|^\beta\big] & \leq C_\beta(|t-s|^{\delta_1\beta} + |x-y|^{\delta_2\beta}),
\end{align}
for every $(t,x), (s,y)\in[0,T]\times \Rd$, and some $C_\gamma$, $C_\beta>0$.

\noindent{\bf (H3)}
$|\sigma(t,x)|>0$, $\omega$-a.s. 

\noindent{\bf (H4)}
Let $\alpha$, $\bar C$, $C_\beta$ and $c$ as in Assumptions \ref{ass:stable} and $0<\eps<t$. We suppose that
\begin{align}
	&\int_{t-\eps}^t\int_\Rd \Infkt{A_t(x)}(s,y)\bar c(s,y) |g(t,s,x,y)|^\alpha\lambda(ds,dy)<\infty, \label{eq:intcond1}\\
	c\eps^{\gamma_0} \leq &\int_{t-\eps}^t\int_\Rd \Infkt{A_t(x)}(s,y) c(s,y) |g(t,s,x,y)|^\alpha \lambda(ds,dy)<\infty, \label{eq:intcond2}
\end{align}
where in \eqref{eq:intcond1}, $\bar c(s,y)=\bar C(s,y)\vee C_0(s,y)$, and \eqref{eq:intcond2} holds for some $\gamma_0>0$.

\noindent Moreover, there exist constants $C,\gamma_ 1,\gamma_2>0$ and $\gamma>\alpha$ such that 
\begin{align}
  \int_{t-\eps}^t\int_\Rd  \Infkt{A_t(x)}(s,y)  \tilde C_\beta(s,y) |g(t,s,x,y)|^\gamma |t-\eps-s|^{\delta_1\gamma}\lambda(ds,dy) &\le C \eps^{\gamma\gamma_1}, \label{alpha}\\
  \int_{t-\eps}^t\int_\Rd \Infkt{A_t(x)}(s,y) \tilde C_\beta(s,y) |g(t,s,x,y)|^\gamma |x-y|^{\delta_2\gamma}\lambda(ds,dy) & \leq C\eps^{\gamma\gamma_2}. \label{beta}
\end{align}
We also assume that there exist constants $C,\gamma_3,\gamma_4>0$ such that for all $\beta\in[1,\alpha)$, if $\alpha\ge 1$, or for $\beta=1$, if $\alpha<1$,
\begin{align} 
  \int_{t-\eps}^t\int_\Rd \Infkt{B_t(x)}(s,y) |h(t,s,x,y)|^\beta |t-\eps-s|^{\delta_1\beta} dyds & \leq C\eps^{\beta\gamma_3}, \label{gamma}\\
	\int_{t-\eps}^t\int_\Rd \Infkt{B_t(x)}(s,y) |h(t,s,x,y)|^\beta |x-y|^{\delta_2\beta} dyds & \leq C\eps^{\beta\gamma_4}, \label{delta}
\end{align}
where $\tilde C_\beta$ is defined as in Lemma \ref{lem:continuity}.

{\bf (H5)}
The set $A_t(x)$ ``reaches $t$", i.e. there is no $\eps>0$ satisfying $A_t(x)\subseteq [0,t-\eps]\times\Rd$.
\end{assumption}

\begin{remark}
\begin{enumerate}[label=(\roman{enumi}),ref=(\roman{enumi})]
	\item By the conditions in {\bf (H4)}, the stochastic integral in \eqref{eq:ambitfield} with respect to the L\'evy basis is well-defined as a random variable in $L^\beta(\Omega)$ for any $\beta\in(0,\alpha)$ (see Lemma \ref{lem:continuity}). 
  \item One can easily derive some sufficient conditions for the assumptions in {\bf (H4)}. Indeed, suppose that
  \[ \int_{t-\eps}^t\int_\Rd  \Infkt{A_t(x)}(s,y) \tilde C_\beta(s,y) |g(t,s,x,y)|^\gamma \lambda(ds,dy) \leq C\eps^{\gamma\bar\gamma_1}, \]
  then \eqref{alpha} holds with $\gamma_1 = \bar\gamma_1+\delta_1$. If in addition, $A_t(x)$ consist
  of points $(s,y)\in[0,t]\times\Rd$ such that $|x-y| \leq |t-s|^\zeta$, for any $s\in[t-\eps,t]$, and for some $\zeta>0$,
  then \eqref{beta} holds with $\gamma_2 = \bar\gamma_1+\delta_2\zeta$. Similarly, one can derive sufficient conditions for \eqref{gamma}, \eqref{delta}.
	\item The assumption {\bf (H5)} is used in the proof of Theorem \ref{thm:densityambit}, where the law of $X(t,x)$ is compared with that of an approximation $X^\eps(t,x)$, which is infinitely divisible. This distribution is well-defined only if $A_t(x)$ is non-empty in the region $[t-\eps,t]\times \Rd$.
	  \item Possibly, for particular examples of ambit sets $A_t(x)$, functions $g, h$, and stochastic processes $\sigma, b$, the Assumptions \ref{ass:stable2} can be relaxed. However, we prefer to keep this formulation.
\end{enumerate}
\end{remark}

We can now state the main theorem of this section.

\begin{theorem}\label{thm:densityambit}
We suppose that the Assumptions \ref{ass:stable} and \ref{ass:stable2} are satisfied and that
  \begin{equation}\label{eq:conditionambit}  
		\frac{\min\{\gamma_1,\gamma_2,\gamma_3,\gamma_4\}}{\gamma_0} > \frac{1}{\alpha}.
  \end{equation}
  Fix $(t,x)\in(0,T]\times\Rd$. Then the law of the random variable $X(t,x)$ defined by \eqref{eq:ambitfield} is absolutely continuous with respect to the Lebesgue measure.
\end{theorem}

\subsection{Two auxiliary results}
In this subsection we derive two auxiliary lemmas. They play a similar role as those in \cite[Sections 5.1 and 5.2]{debusschefournier}, but our formulation is more general.

\begin{lemma}\label{lem:moment}
  Let $\rho=(\rho_{s,y})_{(s,y)\in[0,T]\times\Rd}$ be a family of L\'evy measures and let $\lambda$ be a control measure. Suppose that  Assumption \ref{ass:stable}\ref{itm:stableii} holds. Then for all $\gamma\in(\alpha,2)$ and all $a\in(0,\infty)$
 \begin{equation*}
  \int_{|z|\leq a} |z|^\gamma\rho_{s,y}(dz) \leq C_{\gamma,\alpha}\bar C(s,y)a^{\gamma-\alpha}, \  \lambda- a.s.,
  \end{equation*}
    where $C_{\gamma,\alpha} = 2^{-\gamma+2}\frac{2^{2-\alpha}}{2^{\gamma-\alpha}-1}$. Hence
    \begin{equation*}
    \int_0^T \int_{\Rd}\int_{|z|\le a} |z|^\gamma \rho_{s,y}(dz)\lambda(ds,dy) \le C a^{\gamma-\alpha}.
    \end{equation*}
\begin{proof} 
It is obtained by the following computations:
\begin{align*}
  \int_{|z|\leq a} |z|^\gamma\rho_{s,y}(dz)
  & = \sum_{n=0}^\infty \int_{\{a2^{-n-1}<|z|\leq a2^{-n}\}} |z|^\gamma \rho_{s,y}(dz) \\
  & \leq \sum_{n=0}^\infty (a2^{-n-1})^{\gamma-2} \int_{\{|z|\leq a2^{-n}\}} |z|^2 \rho_{s,y}(dz) \\
  & \leq \bar C(s,y)\sum_{n=0}^\infty (a2^{-n-1})^{\gamma-2} (a2^{-n})^{2-\alpha} \\
  & \leq C_{\gamma-\alpha}\bar C(s,y)a^{\gamma-\alpha}. \qedhere
\end{align*}
\end{proof}
\end{lemma}


The next lemma provides important bounds on the moments of the stochastic integrals. It plays the role of \cite[Lemma 5.2]{debusschefournier} in the setting of this article.

\begin{lemma}\label{lem:continuity}
  Assume that $L$ is a L\'evy basis with characteristic exponent satisfying Assumptions \ref{ass:stable} for some $\alpha\in(0,2)$. Let $H=(H(t,x))_{(t,x)\in[0,T]\times\Rd}$ be a predictable process. Then for all $0<\beta<\alpha<\gamma\leq2$ and for all $0\le s< t\leq s+1$,
  \begin{align} 
  \label{secondlemma}
		&\E\bigg[\bigg|\int_s^t\int_\Rd \Infkt{A_t(x)}(r,y) g(t,r,x,y)H(r,y)L(dr,dy)\bigg|^\beta\bigg] \notag\\
		& \leq C_{\alpha,\beta,\gamma} |t-s|^{\beta/\alpha-1}\notag\\
		&\qquad \times \bigg(\int_s^t\int_\Rd \Infkt{A_t(x)}(r,y) \tilde{C}_\beta(r,y) |g(t,r,x,y)|^\gamma \E\big[|H(u,y)|^\gamma\big]\lambda(dr,dy)\bigg)^{\beta/\gamma},
  \end{align}
  where $\tilde{C}_{\beta}(r,y)$ is the maximum of $\bar{C}(r,y)$, and  $(C_\beta+C_1)(r,y)$ (see Assumptions \ref{ass:stable} for the definitions). 
\begin{proof}
There exists a Poisson random measure $N$ such that for all $A\in\caB(\Rd)$,
\[ L([s,t]\times A) = \int_s^t\int_A\int_{|z|\leq 1} z \tilde N(dr,dy,dz) + \int_s^t\int_A\int_{|z|> 1}zN(dr,dy,dz)\]
(see e.g. \cite[Theorem 4.6]{pedersen}), where  $\tilde N$ stands for the compensated Poisson random measure $\tilde N(ds,dy,dz) = N(ds,dy,dz) - \rho_{s,y}(dz)\lambda(ds,dy)$.
Then we can write
\begin{equation}\label{eq:proofdis}
  \E\bigg[\bigg|\int_s^t\int_\Rd \Infkt{A_t(x)}(r,y) g(t,r,x,y)H(r,y)L(dr,dy)\bigg|^\beta\bigg] \leq C_\beta\big(I^1_{s,t} + I^2_{s,t} + I^3_{s,t}\big),
\end{equation}
with
\begin{align*}
  I^1_{s,t} & := \E\bigg[\bigg|\int_s^t\int_\Rd \Infkt{A_t(x)}(r,y) \int_{|z|\leq (t-s)^{1/\alpha}} zg(t,r,x,y)H(r,y)\tilde N(dr,dy,dz)\bigg|^\beta\bigg] \\
  I^2_{s,t} & := \E\bigg[\bigg|\int_s^t\int_\Rd \Infkt{A_t(x)}(r,y) \int_{(t-s)^{1/\alpha} < |z|\leq 1} zg(t,r,x,y)H(r,y)\tilde N(dr,dy,dz)\bigg|^\beta\bigg] \\
  I^3_{s,t} & := \E\bigg[\bigg|\int_s^t\int_\Rd \Infkt{A_t(x)}(r,y) \int_{|z|> 1} zg(t,r,x,y)H(r,y)N(dr,dy,dz)\bigg|^\beta\bigg]
\end{align*}
To give an upper bound for the first term, we apply first Burkholder's inequality, then the subadditivity of the function $x\mapsto x^{\gamma/2}$ (since the integral is actually a sum), Jensen's inequality, the isometry of Poisson random measures and Lemma \ref{lem:moment}. We obtain,
\begin{align*}
  I^1_{s,t}
  & \leq C_\beta \E\bigg[\bigg|\int_s^t\int_\Rd \Infkt{A_t(x)}(r,y)\int_{|z|\leq (t-s)^{1/\alpha}}\\
  & \qquad \times  |z|^2 |g(t,r,x,y)|^2 |H(r,y)|^2 N(dr,dy,dz)\bigg|^{\beta/2}\bigg] \\
  & \leq C_\beta \E\bigg[\bigg|\int_s^t\int_\Rd \Infkt{A_t(x)}(r,y)\int_{|z|\leq (t-s)^{1/\alpha}}\\
  &\qquad \times |z|^\gamma |g(t,r,x,y)|^\gamma |H(r,y)|^\gamma N(dr,dy,dz)\bigg|^{\beta/\gamma}\bigg] \\
  & \leq C_\beta \bigg(\E\bigg[\int_s^t\int_\Rd \Infkt{A_t(x)}(r,y)\int_{|z|\leq (t-s)^{1/\alpha}}\\
  &\qquad \times |z|^\gamma |g(t,r,x,y)|^\gamma |H(r,y)|^\gamma N(dr,dy,dz)\bigg]\bigg)^{\beta/\gamma} \\
  & = C_\beta \bigg(\E\bigg[\int_s^t\int_\Rd \Infkt{A_t(x)}(r,y)\bigg(\int_{|z|\leq (t-s)^{1/\alpha}} |z|^\gamma \rho_{r,y}(dz)\bigg)\\
  &\qquad \times |g(t,r,x,y)|^\gamma |H(r,y)|^\gamma \lambda(dr,dy)\bigg]\bigg)^{\beta/\gamma} \\
  & \leq C_{\beta}\big(C_{\gamma,\alpha} (t-s)^{(\gamma-\alpha)/\alpha}\big)^{\beta/\gamma}\\
  &\qquad \times \bigg(\int_s^t\int_\Rd \Infkt{A_t(x)}(r,y) \bar C(r,y) |g(t,r,x,y)|^\gamma \E\big[|H(u,y)|^\gamma\big]\lambda(dr,dy)\bigg)^{\beta/\gamma}. 
\end{align*}
Notice that the exponent $(\gamma-\alpha)/\alpha$ is positive.

With similar arguments but applying now Assumption \ref{ass:stable}(i), the second term in \eqref{eq:proofdis} is bounded by
 \begin{align*}
  I^2_{s,t}
  & \leq C_\beta \E\bigg[\bigg|\int_s^t\int_\Rd \Infkt{A_t(x)}(r,y) \int_{(t-s)^{1/\alpha} < |z|\leq 1} |z|^2\\
  &\qquad \times |g(t,r,x,y)|^2 |H(r,y)|^2 N(dr,dy,dz)\bigg|^{\beta/2}\bigg] \\
  & \leq C_\beta \E\bigg[\int_s^t\int_\Rd \Infkt{A_t(x)}(r,y)\int_{(t-s)^{1/\alpha} < |z|\leq 1} |z|^\beta\\
  &\qquad \times |g(t,r,x,y)|^\beta |H(r,y)|^\beta N(dr,dy,dz)\bigg] \\
  & = C_\beta \E\bigg[\int_s^t\int_\Rd \Infkt{A_t(x)}(r,y)\bigg(\int_{(t-s)^{1/\alpha} < |z|\leq 1} |z|^\beta \rho_{r,y}(dz)\bigg)\\
  &\qquad \times|g(t,r,x,y)|^\beta |H(r,y)|^\beta \lambda(dr,dy)\bigg] \\
  & \leq C_{\beta}(t-s)^{(\beta-\alpha)/\alpha} \E\bigg[\int_s^t\int_\Rd \Infkt{A_t(x)}(r,y) C_\beta(r,y)\\
  &\qquad \times |g(t,r,x,y)|^\beta \big[|H(r,y)|^\beta\lambda(dr,dy)\bigg] \\
  & \leq C_{\beta,\gamma}(t-s)^{(\beta-\alpha)/\alpha} \bigg(\int_s^t\int_\Rd \Infkt{A_t(x)}(r,y) C_\beta(r,y)\\
  &\qquad \times |g(t,r,x,y)|^\gamma \E\big[|H(r,y)|^\gamma\big] \lambda(dr,dy)\bigg)^{\beta/\gamma},
\end{align*}
where in the last step we have used H\"older's inequality with respect to the finite measure $C_\beta(r,y)\lambda(dr,dy)$.

Finally, we bound the third term in \eqref{eq:proofdis}. Suppose first that $\beta\leq1$. Using the subadditivity of $x\mapsto x^\beta$ and Lemma \ref{lem:moment} (i) yields
\begin{align*}
  I^3_{s,t}
  & \leq C_\beta \E\bigg[\int_s^t\int_\Rd \Infkt{A_t(x)}(r,y)\int_{|z|> 1} |z|^\beta |g(t,r,x,y)|^\beta |H(r,y)|^\beta N(dr,dy,dz)\bigg] \\
  & \leq C_\beta \E\bigg[\int_s^t\int_\Rd \Infkt{A_t(x)}(r,y)\bigg(\int_{|z|> 1} |z|^\beta \rho_{r,y}(dz)\bigg)\\
  &\qquad \times |g(t,r,x,y)|^\beta |H(r,y)|^\beta \lambda(dr,dy)\bigg] \\
  & \leq C_\beta \E\bigg[\int_s^t\int_\Rd \Infkt{A_t(x)}(r,y) C_\beta(r,y) |g(t,r,x,y)|^\beta |H(r,y)|^\beta \lambda(dr,dy)\bigg] \\
  & \leq C_\beta \bigg(\int_s^t\int_\Rd \Infkt{A_t(x)}(r,y) C_\beta(r,y) |g(t,r,x,y)|^\gamma \E\big[|H(r,y)|^\gamma\big] \lambda(dr,dy)\bigg)^{\beta/\gamma},
\end{align*}
where in the last step we have used H\"older's inequality with respect to the finite measure $C_\beta(r,y)\lambda(dr,dy)$.

Suppose now that $\beta>1$ (which implies that $\alpha>1$). We apply H\"older's inequality with respect to the finite measure $C_1(r,y)\lambda(dr,dy)$ and Assumption \ref{ass:stable}(i) 
\begin{align*}
  I^3_{s,t}
  & \leq 2^{\beta-1}\E\bigg[\bigg|\int_s^t\int_\Rd \Infkt{A_t(x)}(r,y)\int_{|z|> 1} zg(t,r,x,y)H(r,y)\tilde N(dr,dy,dz)\bigg|^\beta\bigg] \\
  & \quad + 2^{\beta-1}\E\bigg[\bigg|\int_s^t\int_\Rd \Infkt{A_t(x)}(r,y)\left(\int_{|z|> 1} |z| \rho_{r,y}(dz)\right)\\
  &\qquad \times g(t,r,x,y)H(r,y)\lambda(dr,dy)\bigg|^\beta\bigg] \\
  & \leq C_\beta \E\bigg[\bigg|\int_s^t\int_\Rd \Infkt{A_t(x)}(r,y)\int_{|z|> 1} |z|^2 |g(t,r,x,y)|^2\\
  &\qquad \times |H(r,y)|^2 N(dr,dy,dz)\bigg|^{\beta/2}\bigg] \\
  & \quad + C_\beta \E\bigg[\bigg|\int_s^t\int_\Rd \Infkt{A_t(x)}(r,y) C_1(r,y) |g(t,r,x,y)|H(r,y)|\lambda(dr,dy)\bigg|^\beta\bigg] \\
  & \leq C_\beta \E\bigg[\int_s^t\int_\Rd \Infkt{A_t(x)}(r,y)\int_{|z|> 1} |z|^\beta |g(t,r,x,y)|^\beta |H(r,y)|^\beta N(dr,dy,dz)\bigg] \\
  & \quad + C_\beta \bigg(\int_s^t\int_\Rd  C_1(r,y)\lambda(dr,dy)\bigg)^{\beta-1}\\
  &\qquad \times \int_s^t\int_\Rd \Infkt{A_t(x)}(r,y) C_1(r,y) |g(t,r,x,y)|^\beta\E\big[|H(r,y)|^\beta\big]\lambda(dr,dy) \\
  & \leq C_\beta \E\bigg[\int_s^t\int_\Rd \Infkt{A_t(x)}(r,y) (C_1(r,y)+C_\beta(r,y))\\
  &\qquad \times |g(t,r,x,y)|^\beta |H(u,y)|^\beta \lambda(dr,dy)\bigg] \\
  & \leq C_\beta \bigg(\int_s^t\int_\Rd \Infkt{A_t(x)}(r,y) (C_1(r,y)+C_\beta(r,y))\\
  &\qquad \times |g(t,r,x,y)|^\gamma \E\big[|H(u,y)|^\gamma\big] \lambda(dr,dy)\bigg)^{\beta/\gamma},
\end{align*}
where in the last step we have used H\"older's inequality with respect to the finite measure $(C_1(r,y) + C_\beta(r,y))\lambda(dr,dy)$. We are assuming $0<t-s\le1$, and $0<\beta<\alpha$. Hence, the estimates on the terms $I^i_{s,t}$, $i=1,2,3$ imply \eqref{secondlemma}
\end{proof}
\end{lemma}


\subsection{Existence of density}
With the help of the two lemmas in the previous subsection, we can now give the proof of Theorem \ref{thm:densityambit}. Fix $(t,x)\in(0,T]\times\Rd$ and let $0<\eps<t$ to be determined later. We define an approximation of the ambit field $X(t,x)$ by
\begin{equation}\label{eq:approxambit}
  X^\eps(t,x) = U^\eps(t,x) + \sigma(t-\eps,x)\int_{t-\eps}^t\int_\Rd \Infkt{A_t(x)}(s,y)g(t,s;x,y)L(ds,dy),  
\end{equation}
where 
\begin{align*}
	U^\eps(t,x) = x_0		& + \int_0^{t-\eps}\int_\Rd \Infkt{A_t(x)}(s,y) g(t,s;x,y)\sigma(s,y)L(ds,dy) \\
											&	+ \int_0^{t-\eps}\int_\Rd \Infkt{B_t(x)}(s,y) h(t,s;x,y)b(s,y)dyds \\
											& + b(t-\eps,x)\int_0^{t-\eps}\int_\Rd \Infkt{B_t(x)}(s,y) h(t,s;x,y)dyds
\end{align*}
Note that $U^\eps(t,x)$ is $\scrF_{t-\eps}$-measurable.

The stochastic integral in \eqref{eq:approxambit} is well defined in the sense of \cite{rajputposinski} and is a random variable having an infinitely divisible distribution. Moreover, the real part of its characteristic exponent is given by
\[ \Re\big(\log \E\big[\exp(i\xi X)\big]\big) = \int_\R \big(1-\cos(\xi z)\big)\rho_f(dz), \]
where
\[ \rho_f(B) = \int_{[0,T]\times\Rd}\int_\R \Infkt{\{zf(s,y)\in B\backslash\{0\}\}}\rho_{s,y}(dz)\lambda(ds,dy). \]

In the setting of this section, the next lemma plays a similar role as  \cite[Lemma 2.3]{sanzsuess3}. It generalizes \cite[Lemma 3.3]{debusschefournier} to the case of L\'evy bases as integrators.

\begin{lemma}\label{lem:existdensstable}
	The Assumptions \ref{ass:stable}, along with  \eqref{eq:intcond1} and \eqref{eq:intcond2} hold. Then, the random variable 
	\[ X:= \int_{t-\eps}^t \int_\Rd \Infkt{A_t(x)}(s,y) g(t,s,x,y)L(ds,dy) \]
	has a $\caC^\infty$-density $p_{t,x,\eps}$, and for all $n\in\N$ there exists a finite constant $C_n>0$ such that $\|p_{t,x,\eps}^{(n)}\|_{L^1(\R)}\leq C_{n,t,x} (\eps^{\gamma_0}\wedge 1)^{-n/\alpha}$.
\begin{proof}
We follow the proof of \cite[Lemma 3.3]{debusschefournier}, which builds on the methods of \cite{schilling-sztonyk-wang}. 
First we show that for $|\xi|$ sufficiently large, and every $t\in(0,T]$, 
\begin{equation}
\label{u&l}
c_{t,x,\eps}|\xi|^\alpha \leq \Re\Psi_{X}(\xi) \leq C|\xi|^\alpha. 
\end{equation}
Indeed, let $r$ be as in Assumption \ref{ass:stable}\ref{itm:stableiii}. Then, for $|\xi| > r$, we have
\begin{align}
\label{lower}
  \Re\Psi_{X}(\xi)
  & = \int_\R \big(1-\cos(\xi z)\big)\rho_f(dz)\notag \\
  & = \int_{t-\eps}^t \int_{\Rd}\lambda(ds,dy) \int_\R \big(1-\cos(\xi z \Infkt{A_t(x)}(s,y) g(t,s,x,y))\big)\rho_{s,y}(dz)\notag\\
  & \geq |\xi|^\alpha\int_{t-\eps}^t\int_\Rd  \Infkt{A_t(x)}(s,y) |g(t,s,x,y)|^\alpha c(s,y)\lambda(ds,dy)\notag \\
  & \ge c_{t,\eps,x} \eps^{\gamma_0}|\xi|^\alpha.
\end{align}
This proves the lower bound in \eqref{u&l} for $|\xi| > r$.

%

In order to prove the upper bound in \eqref{u&l}, we set 
\[ a_{\xi,t,s,x,y} := |\xi|\Infkt{A_t(x)}(s,y)|g(t,s,x,y)| \]
and use the inequality $(1-\cos(x))\leq 2(x^2\wedge1)$ to obtain
\begin{align}\label{real}
  \Re\Psi_X(\xi) 
  & = \int_{t-\eps}^t\int_\Rd \lambda(ds,dy) \int_\R \big(1-\cos(z\xi\Infkt{A_t(x)}(s,y)g(t,s,x,y))\big)\rho_{s,y}(dz)\notag\\
  & \leq 2 \int_{t-\eps}^t\int_\Rd \lambda(ds,dy) \int_\R \big(|z|^2|\xi|^2\Infkt{A_t(x)}(s,y)|g(t,s,x,y)|^2\wedge 1\big)\rho_{s,y}(dz)\notag\\  
  & = 2\int_{t-\eps}^t\int_\Rd \lambda(ds,dy)\int_{|z|\leq a_{\xi,t,s,x,y}^{-1}} |z|^2|\xi|^2\Infkt{A_t(x)}(s,y)|g(t,s,x,y)|^2\rho_{s,y}(dz)\notag \\
  & \quad + 2\int_{t-\eps}^t\int_\Rd \lambda(ds,dy)\int_{|z|\geq a_{\xi,t,s,x,y}^{-1}} \rho_{s,y}(dz).
\end{align}

Then, using Assumption \ref{ass:stable}\ref{itm:stableii}, the first integral in the right-hand side of the last equality 
in \eqref{real} can be bounded as follows:
\begin{align*}
  & \int_{t-\eps}^t\int_\Rd \lambda(ds,dz) |\xi|^2\Infkt{A_t(x)}(s,y)|g(t,s,x,y)|^2 \left(\int_{|z|\leq a_{\xi,t,s,x,y}^{-1}} |z|^2\rho_{s,y}(dz)\right)\\
  & \leq |\xi|^\alpha\int_{t-\eps}^t\int_\Rd \Infkt{A_t(x)}(s,y)|g(t,s,x,y)|^\alpha \bar C(s,y)\lambda(ds,dy)\\
  & \le  C |\xi|^\alpha,
\end{align*}
where in the last inequality, we have used \eqref{eq:intcond1}.

Consider now the last integral in \eqref{real}. By applying Assumption \ref{ass:stable}(i) with $\beta=0$ and \eqref{eq:intcond1}
\begin{align*}
  & \int_{t-\eps}^t\int_\Rd\lambda(ds,dy) \left(\int_{|z|\geq a_{\xi,t,s,x,y}^{-1}} \rho_{s,y}(dz)\right)\\
  & \qquad \le |\xi|^\alpha\int_{t-\eps}^t\int_\Rd\Infkt{A_t(x)}(s,y) C_0(s,y)|g(t,s,x,y)|^\alpha\lambda(ds,dy)\\
  &\qquad \le C |\xi|^\alpha.
\end{align*}
Hence, we have established that
\begin{equation*}
  \Re\Psi_X(\xi) \le C |\xi|^\alpha,
  \end{equation*}
for $|\xi|$ sufficiently large. 

To complete the proof, we can follow the same arguments as in  \cite[Lemma 3.3]{debusschefournier} which relies on the result in \cite[Proposition 2.3]{schilling-sztonyk-wang}. Note that the exponent $\gamma_0$ on the right-hand side of the gradient estimate accounts for the lower bound of the growth of the term in \eqref{eq:intcond2}, which in the case of SDEs is equal to $1$. 
\end{proof}
\end{lemma}

The next lemma shows that the error in the approximation $X^\eps(t,x)$ in \eqref{eq:approxambit} and the ambit field $X(t,x)$ is bounded by a power of $\eps$.

\begin{lemma}\label{lem:approxX}
  Assume that Assumptions \ref{ass:stable} hold for some $\alpha\in(0,2)$ and that $\sigma,b$ are Lipschitz continuous functions. Then, for any $\beta\in(0,\alpha)$, and $\eps\in(0, t\wedge1)$,
  \[ \E\big[|X(t,x)-X^\eps(t,x)|^\beta\big] \leq C_{\beta}\eps^{\beta\left(\frac{1}{\alpha}+\bar\gamma\right)-1}, \]
 where  $\bar\gamma := \min\{\gamma_1,\gamma_2,\gamma_3,\gamma_4\}$. 
\begin{proof}
Clearly,
\begin{align*}
	& \E\big[|X(t,x)-X^\eps(t,x)|^\beta\big] \\
	& \leq C_\beta\E\bigg[\bigg|\int_{t-\eps}^t\int_\Rd \Infkt{A_t(x)}(s,y) g(t,s;x,y)(\sigma(s,y)-\sigma(t-\eps,x))L(ds,dy)\bigg|^\beta\bigg] \\
	& \phantom{\leq} + C_\beta\E\bigg[\bigg|\int_{t-\eps}^t\int_\Rd \Infkt{B_t(x)}(s,y)h(t,s;x,y)(b(s,y)-b(t-\eps,x))dyds\bigg|^\beta\bigg].
\end{align*}
Fix $\gamma\in(\alpha,2]$ and apply Lemma \ref{lem:continuity} to the stochastic process $H(s,y):=\sigma(s,y)-\sigma(t-\eps,x)$, where the arguments $t,\eps, x$ are fixed. We obtain
\begin{align*}
	& \E\bigg[\bigg|\int_{t-\eps}^t\int_\Rd \Infkt{A_t(x)}(s,y) g(t,s;x,y)(\sigma(s,y)-\sigma(t-\eps,x))L(ds,dy)\bigg|^\beta\bigg] \\
	& \leq C_{\alpha,\beta,\gamma}\eps^{\beta/\alpha-1}
	\bigg(\int_{t-\eps}^t\int_\Rd \Infkt{A_t(x)}(s,y)\tilde C_\beta(s,y) \Infkt{A_t(x)}|g(t,s,x,y)|^\gamma\\
	& \qquad \times \E\big[|\sigma(s,y)-\sigma(t-\eps,x)|^\gamma\big]\lambda(ds,dy)\bigg)^{\beta/\gamma} 
	\end{align*}
Owing to hypothesis {\bf (H2)} this last expression is bounded (up to the constant $C_{\alpha,\beta,\gamma}\eps^{\beta/\alpha-1}$) by 
\begin{equation*}
	\bigg(\int_{t-\eps}^t\int_\Rd \Infkt{A_t(x)}(s,y) \tilde C_\beta(s,y) \Infkt{A_t(x)}|g(t,s,x,y)|^\gamma \big(|t-\eps-s|^{\delta_1\gamma} + |x-y|^{\delta_2\gamma}\big)\lambda(ds,dy)\bigg)^{\beta/\gamma} 
	\end{equation*}
The inequality \eqref{alpha} implies 
\begin{align*}
&\eps^{\beta/\alpha-1}\bigg(\int_{t-\eps}^t\int_\Rd \Infkt{A_t(x)}(s,y) \tilde C_\beta(s,y)|g(t,s,x,y)|^\gamma |t-\eps-s|^{\delta_1\gamma}\lambda(ds,dy)\bigg)^{\beta/\gamma}\\ 
&\qquad \le C \eps^{\beta(\frac{1}{\alpha}+\gamma_1)-1},
\end{align*}
and \eqref{beta} yields
\begin{align*}
&\eps^{\beta/\alpha-1}\bigg(\int_{t-\eps}^t\int_{\Rd} \Infkt{A_t(x)}(s,y) \tilde C_\beta(s,y)|g(t,s,x,y)|^\gamma |x-y|^{\delta_2\gamma}\lambda(ds,dy)\bigg)^{\beta/\gamma}\\
&\qquad \le C \eps^{\beta(\frac{1}{\alpha}+\gamma_2)-1}.
\end{align*}
Thus,
\begin{align}
\label{first}
&\E\bigg[\bigg|\int_{t-\eps}^t\int_\Rd \Infkt{A_t(x)}(s,y) g(t,s;x,y)(\sigma(s,y)-\sigma(t,x))L(ds,dy)\bigg|^\beta\bigg]\notag\\
&\qquad  \le C  \eps^{\beta(\frac{1}{\alpha}+[\gamma_1\wedge\gamma_2])-1}.
\end{align}

Assume that $\beta\geq1$ (and therefore $\alpha>1$). H\"older's inequality with respect to the finite measure $h(t,s;x,y)dyds$,  {\bf(H2)}, \eqref{gamma}, \eqref{delta}, imply   
\begin{align*}
	& \E\bigg[\bigg|\int_{t-\eps}^t\int_\Rd \Infkt{B_t(x)}(s,y) h(t,s;x,y)(b(s,y)-b(t-\eps,x))dyds\bigg|^\beta\bigg] \\
	& \leq C_{\beta}\int_{t-\eps}^t\int_\Rd \Infkt{B_t(x)}(s,y) |h(t,s,x,y)|^\beta \E\big[|b(s,y)-b(t-\eps,x)|^\beta\big]dyds \\
	& \leq C_{\beta}\int_{t-\eps}^t\int_\Rd \Infkt{B_t(x)}(s,y) |h(t,s,x,y)|^\beta |t-\eps-s|^{\delta_1\beta} dyds \\
	& \phantom{\leq} + C_{\beta}\int_{t-\eps}^t\int_\Rd \Infkt{B_t(x)}(s,y) |h(t,s,x,y)|^\beta |x-y|^{\delta_2\beta} dyds \\
	& \leq C \eps^{\beta(\gamma_3\wedge\gamma_4)}. 
\end{align*}

Suppose now that $\beta<1$, we use Jensen's inequality and once more, {\bf(H2)}, \eqref{gamma}, \eqref{delta}, to obtain   
 
\begin{align*}
	& \E\bigg[\bigg|\int_{t-\eps}^t\int_\Rd \Infkt{B_t(x)}(s,y) h(t,s;x,y)(b(s,y)-b(t-\eps,x))dyds\bigg|^\beta\bigg] \\
	& \leq \bigg(\E\bigg[\int_{t-\eps}^t\int_\Rd \Infkt{B_t(x)}(s,y) |h(t,s;x,y)||b(s,y)-b(t-\eps,x)|dyds\bigg]\bigg)^\beta \\
	& = \bigg(\int_{t-\eps}^t\int_\Rd \Infkt{B_t(x)}(s,y) |h(t,s;x,y)|\E\big[|b(s,y)-b(t-\eps,x)|\big]dyds\bigg)^\beta \\
	& \leq C\bigg(\int_{t-\eps}^t\int_\Rd \Infkt{B_t(x)}(s,y) |h(t,s;x,y)|\left[|t-s|^{\delta_1}|+|x-y|^{\delta_2}\right] dyds\bigg)^\beta \\
		& \leq C\eps^{\beta(\gamma_3\wedge\gamma_4)}.
\end{align*}	
This finishes the proof.
\end{proof}
\end{lemma}


We are now in a position to prove Theorem \ref{thm:densityambit}.


\begin{proof}[Proof of Theorem \ref{thm:densityambit}]
We consider the inequality
\begin{align}\label{eq:proof41}
	|\E[|\sigma(t,x)|^n\Delta_h^n\phi(X(t,x))]| 
	\leq & \big|\E\big[\big(|\sigma(t,x)|^n - |\sigma(t-\eps,x)|^n\big)\Delta_h^n\phi(X(t,x))\big]\big| \notag\\
		& + \big|\E\big[|\sigma(t-\eps,x)|^n\big(\Delta_h^n\phi(X(t,x)) - \Delta_h^n\phi(X^\eps(t,x))\big)\big]\big| \notag\\
		& + \big|\E\big[|\sigma(t-\eps,x)|^n\Delta_h^n\phi(X^\eps(t,x))\big]\big|.
\end{align}
Fix $\eta\in(0,\alpha\wedge1)$. As in \eqref{eq:bounddeltahn2} we have
\begin{align*}
  & \big|\E\big[\big(|\sigma(t,x)|^n - |\sigma(t-\eps,x)|^n\big)\Delta_h^n\phi(X(t,x))\big]\big| \\
  & \leq C_n\|\phi\|_{\caC_b^\eta}|h|^\eta \E\big[\big||\sigma(t,x)|^n - |\sigma(t-\eps,x)|^n\big|\big]\big|
\end{align*}
Now we proceed as in \eqref{eq:boundpolynomial} using the finiteness of the moments of $\sigma(t,x)$ stated in Hypothesis {\bf (H1)}, and {\bf (H2)}. Then for all $\gamma\in(\alpha,2]$ we have
\begin{align*}
  & \E\big[\big||\sigma(t,x)|^n-|\sigma(t-\eps,x)|^n\big|\big] \notag\\
  & = \E\bigg[ \big|\sigma(t,x)-\sigma(t-\eps,x)\big|\sum_{j=0}^{n-1} |\sigma(t,x)|^j|\sigma(t-\eps,x)|^{n-1-j}\bigg] \notag\\
  & \leq C\Big(\E\big[\big|\sigma(t,x) - \sigma(t-\eps,x)\big|^\gamma\big]\Big)^{1/\gamma}\\
  &\qquad\times \bigg(\E\bigg[\bigg(\sum_{j=0}^{n-1} |\sigma(t,x)|^j|\sigma(t-\eps,x))|^{n-1-j}\bigg)^{\gamma/(\gamma-1)}\bigg]\bigg)^{1-1/\gamma} \notag\\
   & \leq C_n\eps^{\delta_1}.
\end{align*}
Therefore
\begin{equation}\label{eq:prf1}
  \big|\E\big[\big(|\sigma(t,x)|^n - |\sigma(t-\eps,x)|^n\big)\Delta_h^n\phi(X(t,x))\big]\big| \leq C_n\|\phi\|_{\caC_b^\eta}|h|^\eta \eps^{\delta_1}.
\end{equation}

Consider the inequality $\|\Delta_h^n\phi\|_{\caC^\alpha_b}\leq C_n\|\phi\|_{\caC^\alpha_b}$, and apply  H\"older's inequality with some $\beta\in(\eta,\alpha)$ to obtain
\begin{align}\label{eq:prf2}
  & \big|\E\big[|\sigma(t-\eps,x)|^n(\Delta_h^n\phi(X(t,x)) - \Delta_h^n\phi(X^\eps(t,x)))\big]\big| \notag\\
  & \leq C_n\|\phi\|_{\caC^\eta_b}\E\big[|\sigma(t-\eps,x)|^n|X(t,x)-X^\eps(t,x)|^\eta\big] \notag\\
  & \leq C_n\|\phi\|_{\caC^\eta_b}\big(\E\big[|X(t,x)-X^\eps(t,x)|^{\beta}\big]\big)^{\eta/\beta}\big(\E\big[|\sigma(u(t-\eps,0))|^{n\beta/(\beta-\eta)}\big]\big)^{1-\eta/\beta} \notag\\
  & \leq C_{n,\beta}\|\phi\|_{\caC^\eta_b}\eps^{\eta\left(\frac{1}{\alpha}+\bar\gamma\right)-\frac{\eta}{\beta}},
\end{align}
where $\bar\gamma := \min\{\gamma_1,\gamma_2,\gamma_3,\gamma_4\}$, and we have applied Lemma \ref{lem:approxX}.

Conditionally to $\scrF_{t-\eps}$, the random variable
\[ \int_{t-\eps}^t\int_\Rd \Infkt{A_t(x)}(s,y)g(t,s;x,y)L(ds,dy) \]
has an infinitely divisible law and a $\caC^\infty$-density $p_{t,x\eps}$ for which a gradient estimate holds (see Lemma \ref{lem:existdensstable}). Then, by a discrete integration by parts, and owing to {\bf (H3)},
\begin{align*}
	& \big|\E\big[|\sigma(t-\eps,x)|^n\Delta_h^n\phi(X^\eps(t,x))\big]\big| \\
  & = \bigg|\E\bigg[\int_\R |\sigma(t-\eps,x)|^n\Delta_h^n\phi(U_t^\eps + \sigma(t-\eps,x)y)p_{t,x,\eps}(y)dy\bigg]\bigg| \\
  & = \bigg|\E\bigg[\int_\R |\sigma(t-\eps,x)|^n\phi(U_t^\eps + \sigma(t-\eps,x)y)\Delta_{-\sigma(t-\eps,x)^{-1}h}^n p_{t,x,\eps}(y)dy\bigg]\bigg| \\
  & \leq \|\phi\|_\infty \E\bigg[|\sigma(t-\eps,x)|^n\int_\R\big|\Delta_{-\sigma(t-\eps,x)^{-1}h}^np_{t,x,\eps}(y)\big|dy\bigg].
\end{align*}  
From Lemma \ref{lem:existdensstable} it follows that
\begin{align*}
	\int_\R\big|\Delta_{-\sigma(t-\eps,x)^{-1}h}^np_{t,x,\eps}(y)\big|dy
  & \leq C_n |\sigma(t-\eps,x)|^{-n}|h|^n \|p^{(n)}_{t,x,\eps}\|_{L^1(\R)} \\
  & \leq C_n |\sigma(t-\eps,x)|^{-n}|h|^n \eps^{-n\gamma_0/\alpha},
\end{align*}  
which yields
\begin{equation}\label{eq:prf3}
  \big|\E\big[|\sigma(t-\eps,x)|^n\Delta_h^n\phi(X^\eps(t,x))\big]\big| \leq C_n\|\phi\|_{\caC^\eta_b}|h|^n \eps^{-n\gamma_0/\alpha},
\end{equation}
because $\|\phi\|_{\infty}\leq\|\phi\|_{\caC^\eta_b}$.

The estimates \eqref{eq:proof41}, \eqref{eq:prf1}, \eqref{eq:prf2} and \eqref{eq:prf3} imply
\begin{align*}
  |\E[|\sigma(t,x)|^n\Delta_h^n\phi(X(t,x))]| 
  \leq C_{n,\beta}\|\phi\|_{\caC_b^\eta} \big(|h|^\eta\eps^{\delta_1} + \eps^{\eta\left(\frac{1}{\alpha}+\bar\gamma\right)-\frac{\eta}{\beta}} + |h|^n \eps^{-n\gamma_0/\alpha}\big)
\end{align*}
Set $\eps=\frac{t}{2}|h|^\rho$, with $|h|\le 1$ and 
\begin{equation*}
\rho\in\left(\frac{\alpha\beta}{\beta+\alpha\beta\bar\gamma-\alpha}, \frac{\alpha(n-\eta)}{n\gamma_0}\right).
\end{equation*}
Notice that, since $\lim_{n\to\infty}\frac{\alpha(n-\eta)}{n\gamma_0}= \frac{\alpha}{\gamma_0}$, for $\beta$ close to $\alpha$ and $\gamma_0$ as in the hypothesis, this interval is nonempty. Then, easy computations show that with the  choices of $\eps$ and $\rho$, one has
\begin{equation*}
|h|^\eta\eps^{\delta_1} + \eps^{\eta\left(\frac{1}{\alpha}+\bar\gamma\right)-\frac{\eta}{\beta}} + |h|^n \eps^{-n\gamma_0/\alpha}\le 3|h|^\zeta,
\end{equation*}
with $\zeta>\eta$. Hence, with Lemma \ref{lem:existencedensity} we finish the proof of the theorem.

\end{proof}

\begin{remark}
\begin{enumerate}[label=(\roman{enumi}),ref=(\roman{enumi})]
  \item If $\sigma$ is bounded away from zero, then one does not need to assume the existence of moments of sufficiently high order. In this case one can follow the strategy in \cite{sanzsuess3}.
  \item The methodology used in this section is not restricted to pure-jump stable-like noises. One can also adapt it to the case of Gaussian space-time white noises.
\end{enumerate}
\end{remark}

\noindent{\bf Acknowledgement.} The second author wants to thank Andreas Basse-O'Connor for a very fruitful conversation on L\'evy bases.


\begin{thebibliography}{10}

\bibitem{ballyclement1}
V.~Bally and E.~Cl\'ement.
\newblock Integration by parts formulas and applications to equations with jumps.
\newblock{\em Probab. Theory Related Fieds}, 151: 613-657 (2011).

\bibitem{ballyclement2}
V.~Bally and E.~Cl\'ement.
\newblock Integration by parts formulas with respect to jump times and stochastic differential equations.
\newblock {\em Stochastic Analysis 2010, D.O. Crisan (Ed.)}. Springer 2011.

\bibitem{ballyfournier}
V.~Bally and N.~Fournier.
\newblock Regularization properties of the 2D homogeneous Boltzmann equation without cutoff.
\newblock{\em Probab. Theory Related Fieds}, 151, 659-704 (2011).

\bibitem{ballycaramelino}
V.~Bally and L.~Caramellino.
\newblock Convergence and regularity of probability laws by using an interpolation method.
\newblock {\em arXiv:1409.3118v1}.

\bibitem{benth}
O.~Barndorff-Nielsen, F.~Benth, and A.~Veraart.
\newblock Recent advances in ambit stochastics with a view towards
  tempo-spatial stochastic volatility/intermittency.
\newblock {\em Banach Center Publications}, 2014.

\bibitem{bns}
O.~E. Barndorff-Nielsen and J.~Schmiegel.
\newblock L\'evy-based tempo-spatial modelling; with applications to
  turbulence.
\newblock {\em Uspekhi Mat. NAUK}, 59:65--91, 2004.

\bibitem{chongkluppelberg}
C.~Chong and C.~Kl\"uppelberg.
\newblock Integrability conditions for space-time stochastic integrals: Theory
  and Applications.
\newblock {\em arXiv:1303.2468}, 2013.

\bibitem{conusdalang}
D.~Conus and R.~C. Dalang.
\newblock The non-linear stochastic wave equation in high dimensions.
\newblock {\em Electronic Journal of Probability}, 13:629--670, 2008.

\bibitem{dalang}
R.~C. Dalang.
\newblock {Extending Martingale Measure Stochastic Integral with Applications
  to Spatially Homogeneous SPDEs}.
\newblock {\em Electronic Journal of Probability}, 4:1--29, 1999.

\bibitem{debusschefournier}
A.~Debussche and N.~Fournier.
\newblock {Existence of densities for stable-like driven SDEs with H{\"o}lder
  continuous coefficients}.
\newblock {\em Journal of Functional Analysis}, 264(8):1757--1778, 2013.

\bibitem{debusscheromito}
A.~Debussche and M.~Romito.
\newblock {Existence of densities for the 3D Navier-Stokes equations driven by
  Gaussian noise}.
\newblock {\em Probability Theory and Related Fields}, 158(3-4):575--596, 2014.

\bibitem{fournierprintems}
N.~Fournier and J.~Printems.
\newblock Absolute continuity for some one-dimensional processes.
\newblock {\em Bernoulli}, 16(2):343--360, 2010.

\bibitem{malliavin1}
P.~Malliavin.
\newblock{Stochastic Calculus of Variation and Hypoelliptic Operators}.
\newblock{\em In: Proc. Inter. Symp. on Stoch. Diff. Equations, Kyoto 1976}.
\newblock Wiley, 1978, pp. 195-263.

\bibitem{pedersen}
J.~Pedersen.
\newblock The L\'evy-It\^o decomposition of an independently scattered random
  measure.
\newblock 2003.

\bibitem{podolskij}
M.~Podolskij.
\newblock Ambit fields: survey and new challenges.
\newblock {\em arXiv:1405.1531}, 2014.

\bibitem{rajputposinski}
B.~Rajput and J.~Rosinski.
\newblock {Spectral representation of infinitely divisible distributions}.
\newblock {\em Probability Theory and Related Fields}, 82:451--487, 1989.

\bibitem{quersanz1}
Quer-Sardanyons, L. and Sanz-Sol{\'e}, M.:
Absolute Continuity of the Law of the Solution to the 3-dimensional Stochastic Wave Equation.
\emph{Journal of Functional Analysis}, 206, (2004), 1--32.

\bibitem{quersanz2}
Quer-Sardanyons, L. and Sanz-Sol{\'e}, M.:
A stochastic wave equation in dimension 3: smoothness of the law.
\emph{Bernoulli}, 10, (2004), 165--186.


\bibitem{sanzbook}
M.~Sanz-Sol{\'e}.
\newblock {\em {Malliavin Calculus with Applications to Stochastic Partial
  Differential Equations}}.
\newblock EPFL Press, 2005.


\bibitem{sanzsuess3}
M.~Sanz-Sol\'e and A.~S\"u\ss.
\newblock Absolute continuity for SPDEs with irregular fundamental solution.
\newblock {\em arXiv:1409.8031}, 2014.

\bibitem{schilling-sztonyk-wang}
R.~L. Schilling, P.~Sztonyk, and J.~Wang.
\newblock Coupling property and gradient estimates for L\'evy processes via the
  symbol.
\newblock {\em Bernoulli}, 18:1128--1149, 2012.

\bibitem{triebel}
H.~Triebel.
\newblock {\em {Theory of Function Spaces}}.
\newblock Birkh\"a{}user, 1983.

\bibitem{walsh}
J.~B. Walsh.
\newblock {An Introduction to Stochastic Partial Differential Equations}.
\newblock In {\em {Ecole d'\'et\'e de Probabilites de Saint Flour XIV, 1984}},
  volume 1180 of {\em {Lecture Notes in Math}}. Springer, 1986.
\end{thebibliography}
\end{document}